\documentclass[12pt]{amsart}
\usepackage{latexsym} 
  \usepackage[all]{xy}
  \usepackage{amsfonts} 
  \usepackage{amsthm} 
  \usepackage{amsmath} 
  \usepackage{amssymb}
  \usepackage{pifont}  
  \usepackage{enumerate}
  \usepackage{dcolumn}
  \usepackage{comment}
\usepackage{hyperref}  
\newcolumntype{2}{D{.}{}{2.0}}
\usepackage{graphicx}
  \xyoption{2cell}

 \usepackage{tikz}
\usetikzlibrary{arrows} 
\numberwithin{equation}{section}

    \def\<{{\langle}}
    \def\>{{\rangle}}
    
    \def\note#1{{}}

    \def\note#1{}

    \def\beq{\begin{equation}}
    \def\eeq{\end{equation}}

    \def \eE{{\eps_\cE}}

    \def\cten#1{\raise-.2cm\hbox{$\stackrel{\displaystyle\square}
{\scriptscriptstyle{#1}}$}}

    \def\CC{\mathbb{C}}
      \def\ZZ{\mathbb{Z}}
        \def\eE{\mathcal{E}}
    
  \def\QQ{\mathbb{Q}}

     \def\FF{\mathbb{F}}
     \def\PP{\mathbb{P}}

     \newcounter{zlist}

  \newcounter{blist}

\newtheorem{proposition}{Proposition}[section]
\newtheorem{lemma}[proposition]{Lemma}

\newtheorem{theorem}[proposition]{Theorem}
\theoremstyle{definition}

\newtheorem{example}[proposition]{Example}

\theoremstyle{remark} 
\newtheorem{remark}[proposition]{Remark}

\begin{document}
\title[{On the algebra of elliptic curves}]{On the algebra of elliptic curves}

   \author[Brzezi\'nski]{Tomasz Brzezi\'nski}
\address{
Department of Mathematics, Swansea University, 
Swansea University Bay Campus,
Fabian Way,
Swansea,
  Swansea SA1 8EN, U.K.\ \newline \indent
Faculty of Mathematics, University of Bia{\l}ystok, K.\ Cio{\l}kowskiego  1M,
15-245 Bia\-{\l}ys\-tok, Poland}
\email{T.Brzezinski@swansea.ac.uk}

     \date{\today}
   \begin{abstract}
 It is argued that a nonsingular elliptic curve admits a natural or fundamental abelian heap structure uniquely determined by the curve itself. It is shown that the set of complex analytic or rational functions from a nonsingular elliptic curve to itself is a truss arising from endomorphisms of this heap.
   \end{abstract}
   \maketitle

\section{Introduction}
It is well known that a nonsingular complex elliptic curve $\eE: y^2 = 4x^3-g_2x-g_3$ has a natural additive group structure. On the one hand this structure can be understood as arising from the identification of $\eE$ as the quotient $\CC/\Lambda(\omega_1,\omega_2)$, where $\Lambda (\omega_1,\omega_2)= \ZZ\omega_1  +\ZZ\omega_2$, $\Im(\omega_1/\omega_2)>0$ is the lattice of periods: Since $\Lambda(\omega_1,\omega_2)$ is an additive subgroup of $\CC$, the quotient inherits the addition from that of complex numbers. On the other hand the addition can be defined geometrically as follows (see e.g.\ \cite[Section~III.2]{Sil:ari} or \cite[pp.~12--14]{Hus:ell}, where this addition is very suggestively called a {\em chord-tangent law}). By the B\'ezout theorem, a line intersects $\eE$ in three points (counted with multiplicities). Thus a line through points $A$ and $B$ on $\eE$ intersects the curve in the third point which is declared to be $-(A+B)$. Reflecting this point through the $x$-axis we obtain another point of $\eE$ that gives $A+B$.  If the line happens to be tangent to, say, $A$ then $A+B=-A$, while for a point $A$ of multiplicity three, $A+A=-A$. Both constructions make a choice of the neutral point for this operation. While it might be clear why the zero complex number should be the zero of the induced operation (after all zero plays a special role in the usual arithmetic of complex numbers), why the point in infinity on the curve should have this privileged position might not be so transparent; when the curve is embedded in the projective plane $\CC\PP^2$ the point $[0:1:0]$ is no different for any others. Of course, the geometric construction can be repeated by fixing any point as the zero of the operation (the drawing of lines  and description of intersections become a bit more complicated then, see e.g.\ \cite[p.~14]{Hus:ell} or \cite[Section~5.7]{NivZuc:int}), but the fact that a choice of this point has to be made in the first place raises a question of dependence of the structure on this choice rather than the curve alone.  In this note we argue that it is more natural to consider a ternary algebraic structure on an elliptic curve and first liberate oneself from making any choices of special points, and second interpret holomorphic (or rational in the case of a general field) endomorphisms\footnote{By the term {\em endomorphism of an elliptic curve} we mean an analytic in the complex and rational in the general case function from the curve to itself without requesting preservation of any points. For those that preserve a distinguished point we use the term {\em isogeny} as in \cite[Chapter~III.4]{Sil:ari}.} of the curve as special endomorphisms of this structure that combine together into an object with two operations similar to albeit substantially different from a ring. In this way one can deal with {\bf all} endomorphisms of an elliptic curve not only those that fix an arbitrarily chosen point (isogenies).

\section{The result}\label{sec.result}
Given three points $A,B,C$ on an elliptic curve $\eE$ we determine the fourth point $[A,B,C]$ as follows. Suppose that the line through $A$ and $C$ intersects the curve at a point $D\neq B$. Then $[A,B,C]$ is the point of intersection of $\eE$ with the line through $D$ and $B$ (see Figure 1). If it happens that $D=B$, then $[A,B,C]$ is the intersection point of the line tangent to $B$ with the curve $\eE$. If $A=C$, then $D$ is the intersection point of the line tangent to $\eE$ at $A$ with $\eE$. If $B=C$ then this construction gives $[A,C,C] = A$. Similarly, if  $B=A$ then $[A,A,C] =C$. 
\begin{figure}
\includegraphics[scale=.4]{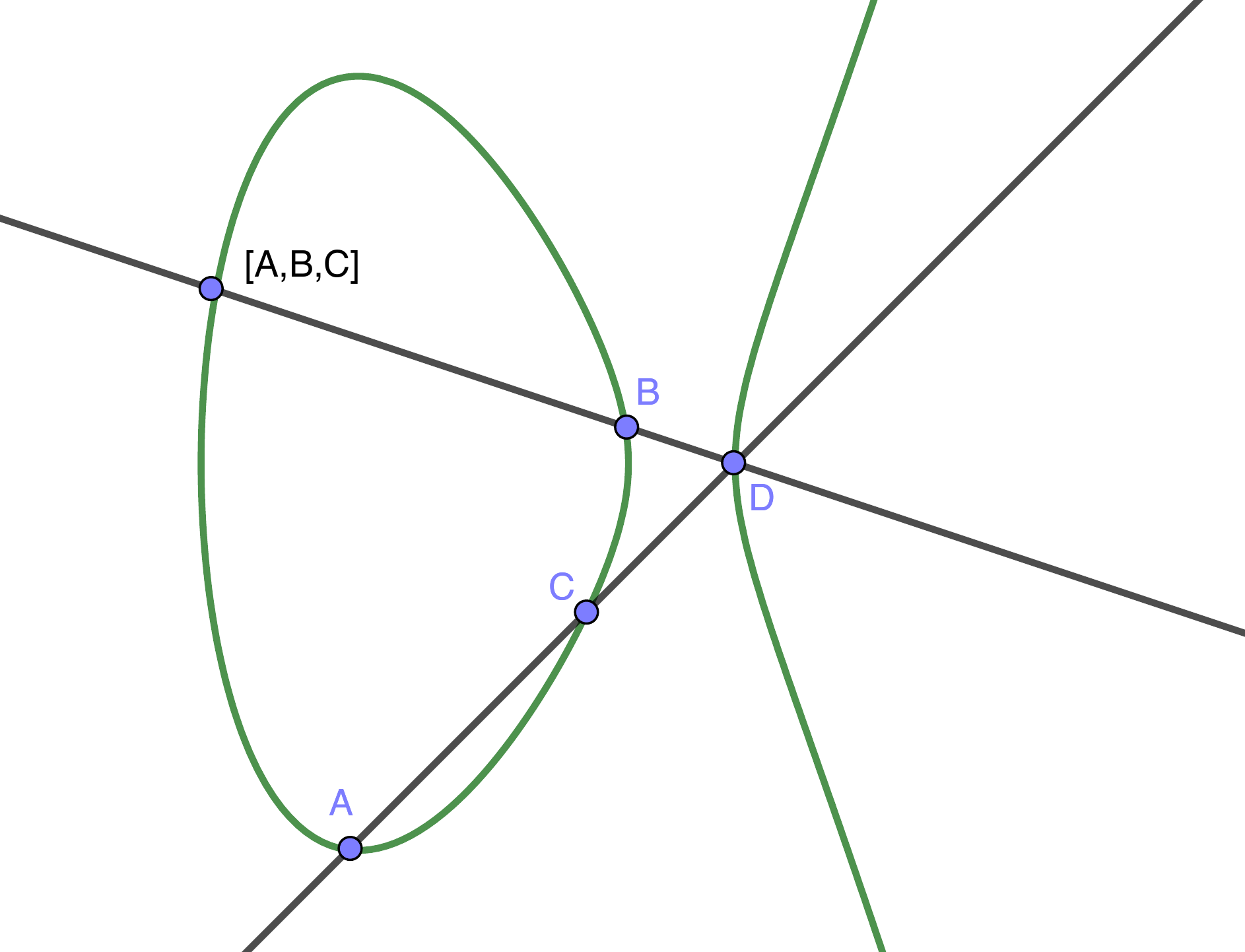}
\caption{Construction of the heap operation}
\end{figure}

The just described operation $(A,B,C)\longmapsto [A,B,C]$ defines an {\em abelian heap} structure on $\eE$ (see \cite{Pru:the}, \cite{Bae:ein}). That is, for all points $A,B,C,D,E$ on $\eE$, 
$$
\begin{aligned} 
{}[[A,B,C],D,E]&=[A,B,[C,D,E]],\\
 [A,B,B] = [B,B,A] =A \;\; &\mbox{and}\;\; [A,B,C]=[C,B,A].
 \end{aligned}
 $$
  We denote this heap by $H(\eE)$.

This geometric construction can also be expressed either analytically or purely algebraically.  Let  $\eE = \CC/\Lambda(\omega_1,\omega_2)$ and let $\wp$ be the Weierstrass function associated with the lattice $\Lambda(\omega_1,\omega_2)$. Take any  $A,B,C \in \eE$ and  view them in $\CC^2$ as  $A=(\wp(a),\wp'(a))$, $B=(\wp(b),\wp'(b))$, $C=(\wp(c),\wp'(c))$, where $a,b,c\in \CC$. Choose $d\in \CC$ such that $a+c+d \in \Lambda(\omega_1,\omega_2)$. The corresponding point $D=(\wp(d),\wp'(d))\in \eE$ is collinear with $A$ and $C$. Let $\langle a,b,c\rangle $ denote any complex number  such that $b+d+\langle a,b,c\rangle \in \Lambda(\omega_1,\omega_2)$. Then
$$
[A,B,C] = (\wp(\langle a,b,c\rangle),\wp'(\langle a,b,c\rangle)).
$$ 
From the algebraic point of view since $\Lambda(\omega_1,\omega_2)$ is a subgroup of the additive group of $\CC$ (and hence a sub-heap of the associated heap),   writing $A:= [a]\in \CC/\Lambda(\omega_1,\omega_2)$ for the class of $a\in \CC$ etc., we find
$$
[A,B,C] = [{a-b+c}].
$$

Note that although analytically number $\langle a,b,c\rangle \in \CC$ is not defined uniquely, algebraically $[a,b,c] := a-b+c$ is. Hence $H(\eE)$ is a {\em bona fide} quotient of the heap structure on the additive group of $\CC$.

The following lemma is a standard result in the theory of elliptic curves (see e.g.\ \cite[Proposition~IV~4.18]{Har:alg} or \cite[Chapter~I~Theorem~4.1]{Lan:ell}).

\begin{lemma}\label{lem.hol}
All (analytic) endomorphisms of a nonsingular elliptic curve $\eE:=\CC/\Lambda(\omega_1,\omega_2)$ are quotients of analytic functions
$$
f_{a,b}: \CC\longrightarrow \CC, \qquad z\longmapsto az+b,
$$
where $a,b\in \CC$ and $a\Lambda(\omega_1,\omega_2)\subseteq \Lambda(\omega_1,\omega_2)$.
\end{lemma}
\begin{proof}
Any analytic function on $\eE$ corresponds to an analytic function $f:\CC\to\CC$ with periods $\omega_1$ and $\omega_2$ modulo $\Lambda(\omega_1,\omega_2)$. Thus,  for all $\alpha\in \Lambda(\omega_1,\omega_2)$, there is an analytic (hence continuous) function
$$
f_\alpha: \CC \longrightarrow \Lambda(\omega_1,\omega_2), \qquad z \longmapsto  f(z+\alpha) - f(z).
$$
Since $\Lambda(\omega_1,\omega_2)$ is discrete, $f_\alpha$ is constant, which implies that 
$$
f'(z+\alpha) = f'(z), \qquad \mbox{for all $\alpha\in \Lambda(\omega_1,\omega_2)$}.
$$
In other words the derivative $f'$ is a doubly-periodic function on $\CC$, and hence it is fully determined by its values on the fundamental parallelogram of $\Lambda(\omega_1,\omega_2)$ with vertices, say, $0$, $\omega_1$, $\omega_2$, $\omega_1+\omega_2$. Since the latter is compact, $f'$ is bounded and thus by Liouville's theorem $f'(z)=a$, for some $a\in \CC$. Therefore, $f(z) =az +b$, for some $a,b\in \CC$. The constraints $f(z+\alpha) - f(z)\in \Lambda(\omega_1,\omega_2)$ yield $a\alpha \in \Lambda(\omega_1,\omega_2)$, for all $\alpha \in \Lambda(\omega_1,\omega_2)$ as stated.
\end{proof}

Recall from \cite{Brz:tru}, \cite{Brz:par} that a {\em truss} is an abelian heap $T$ together with an associative multiplication denoted by juxtposition that distributes over the ternary heap operation $[-,-,-]$, that is, for all $a,b,c,d\in T$,
$$
a[b,c,d] = [ab,ac,ad], \qquad [a,b,c]d = [ad,bd,cd].
$$
The main result of this note is contained in the following 
\begin{theorem}\label{thm.main}
Let $\eE=\CC/\Lambda(\omega_1,\omega_2)$ be a nonsingular elliptic curve. Endomorphisms of $\eE$ are endomorphisms of the heap $H(\eE)$. Consequently the set of all endomorphisms of $\eE$ forms a (unital) truss $T(\eE)$ with the ternary structure inherited from that of $H(\eE)$, that is,
$$
[f,g,h](A) = [f(A),g(A),h(A)], \qquad \mbox{for all $A\in \eE$},
$$
and with the multiplication given by composition.
\end{theorem}
\begin{proof}
By Lemma~\ref{lem.hol} every analytic endomorphism of $\eE$ arises as the quotient of $f_{a,b}(z) = az+b$, $a\Lambda(\omega_1,\omega_2) \subseteq \Lambda(\omega_1,\omega_2)$. Each of these functions is an endomorphism of the heap of $\CC$, i.e.\
$$
\begin{aligned}
f_{a,b}([z,z',z'']) &= f(z-z'+z'') \\
&= (az + b) - (az' + b)+ (az'' + b) = [f_{a,b} (z),f_{a,b}(z'),f_{a,b}(z'')].
\end{aligned}
$$
Since $\Lambda$ is an abelian subgroup and hence also a sub-heap of $\CC$, the functions $f_{a,b}$ descend to endomorphisms of $H(\eE)$ and thus they inherit the structure of a heap as described. 

Obviously, the composition of two endomorphisms of $\eE$ is again an endomorphism. Explicitly, if $f$ corresponds to $f_{a,b}$ and $g$ corresponds to $f_{a',b'}$, then $f\circ g$ corresponds to 
$f_{aa', ab'+b}$.  By the definition of the heap structure $[-,-,-]$ on endomorphisms of $H(\eE)$, the composition right-distributes over $[-,-,-]$ and it left distributes by the fact that analytic endomorphisms of $\eE$ preserve the heap operation of $H(\eE)$.
\end{proof}

\begin{remark}\label{rem.noring}
It might be worth pointing out that the truss $T(\eE)$ is not arising from a ring. Any abelian heap can be converted to an abelian group by retracting it at the middle term in the ternary operation (see e.g.\ \cite[Section~2.2]{Brz:par}). Fixing different points on a curve leads to different albeit isomorphic groups, with the isomorphism provided by  translation.\footnote{The reader interested in the explicit description of this isomorphism in the case of elliptic curves might like to consult \cite[Theorem~5.22]{NivZuc:int}.}  As explained in \cite[Lemma~3.9]{Brz:par} to retract $T(\eE)$ into a ring one would need to have morphism $\Theta:\eE\to \eE$ such that $f\circ \Theta = \Theta \circ f = \Theta$ for all $f\in T(\eE)$. The condition $\Theta \circ f = \Theta$ implies that $\Theta$ must be a constant function, say $\Theta: A\mapsto O$ for a fixed point $O\in \eE$. 
The condition $f\circ \Theta =  \Theta$ now implies that $f(O) = O$, for all morphisms $f$ of $\eE$. This is not possible, as taking a morphism $f$ corresponding to $f_{0,b}$, where $b\not\in O$  or equivalently $O\neq (\wp(b), \wp'(b))$, one immediately obtains that $f(O)\neq O$. 

We can retract the ternary heap operation on $\eE$ at $O$ to the abelian group operation $A+B = [A,O,B]$, for which $O$ is the neutral element. Then $-A = [O,A,O]$ and $[A,B,C] = A-B+C$. The induced group structure on the set of endomorphisms has $\Theta$ as the neutral element. The composition of endomorphisms right distributes over this addition, but the truss left distributive law yields, for all endomorphisms $f,g,h$ and $A\in \eE$,
$$
f\circ (g+h)(A) = f\circ (g -\Theta+h)(A) =  f\circ g(A) - f(O) + f\circ h(A), 
$$
and thus $T(\eE)$ with this (or any other for that matter) addition is not a ring (unless $f(O) =O$ for all endomorphisms $f$ of $\eE$, which as argued above cannot be the case). \hfill $\triangle$
\end{remark}

\begin{remark}\label{rem.ring}
Notwithstanding Remark~\ref{rem.noring}, as explained in \cite[Theorem~4.3]{AndBrz:ide}  any truss, and hence also $T(\eE)$, gives rise to a heap or family of isomorphic rings. For any $A\in \eE$, let $c_A $ denote the constant function $c_A: \eE\to \eE$, $B\mapsto A$ and let us fix a point $O\in \eE$. The abelian group $T(\eE)$ with addition  $+:=[-,c_O,-]$ admits the associative multiplication 
$$
f\bullet g = f\circ g - c_{f(O)},
$$
that distributes over the addition, and hence makes the set of all endomorphisms of $\eE$ a (non-unital) ring. It is clear, however, that this conversion of a truss into a ring requires one to make a choice of an element of the curve. \hfill $\triangle$
\end{remark}

\begin{remark}\label{rem.cross}
In view of the product of endomorphisms described in the proof of Theorem~\ref{thm.main}, $T(\eE)$ can be seen as the extension of a ring by a module \cite[Theorem~4.2]{BrzRyb:con}. Let $\eE = \CC/\Lambda(\omega_1,\omega_2)$ and let $R(\omega_1,\omega_2)$ be the ring of all complex numbers $r$ such that $r \Lambda(\omega_1,\omega_2) \subseteq \Lambda(\omega_1,\omega_2)$. Then $R(\omega_1,\omega_2)$ acts on the heap $H(\eE)$ by the (analytic) formula: for all $r\in R(\omega_1,\omega_2)$ and $A=(\wp(a),\wp'(a))$,
$$
rA := (\wp(ra),\wp'(ra)).\footnote{Thus, in particular, if $r \Lambda(\omega_1,\omega_2) =\Lambda(\omega_1,\omega_2)$, then $rA := (r^{-2}\wp(a),r^{-3}\wp'(a))$ by the homogeneity property of the Weiestrass $\wp$-function.}
$$
Since, for all $r\in R(\omega_1,\omega_2)$,  $r\Lambda(\omega_1,\omega_2)\subseteq\Lambda(\omega_1,\omega_2)$, multiplying by elements of $R(\omega_1,\omega_2)$ preserves the collinearity of points. In consequence, 
$$
r[A,B,C] = [rA,rB,rC], \qquad (r-s+t)A = [rA,sA,tA],
$$
for all $r,s,t\in R(\omega_1,\omega_2)$, $A,B,C \in \eE$. Therefore, $\eE$ is a module over the truss $T(R(\omega_1,\omega_2))$ \cite{Brz:par}, where  $T(R(\omega_1,\omega_2))$ has the same multiplication as that in  $R(\omega_1,\omega_2)$ and the induced abelian heap structure $[r,s,t] =r-s+t$.

Following  \cite[Theorem~4.2]{BrzRyb:con} we can now fix any $O = (\wp(o),\wp'(o))$, and define the truss $T(\omega_1,\omega_2)$ built on the set $R(\omega_1,\omega_2)\times \eE$ with the Cartesian product heap operation
$$
[(r,A),(s, B),(t,C)] = \Big(r-s+t, [A,B,C]\Big),
$$
and multiplication
$$
(r,A)(s,B) = (rs, [A,rO,rB]).
$$
In view of the proof of Theorem~\ref{thm.main}, $T(\eE)\cong T(\omega_1,\omega_2)$. 

Viewed algebraically, that is when the heap structure on $\eE$ arises from the quotient of $\CC$ by its subgroup and hence sub-heap $\Lambda(\omega_1,\omega_2)$, and fixing $O = [0] = \Lambda(\omega_1,\omega_2)$, the multiplication in $T(\omega_1,\omega_2)$ takes a simpler form
$$
(r,[a])(s,[b]) = (rs, [a+rb]). 
$$
\hfill $\triangle$
\end{remark}

Up to isomorphism complex elliptic curves can be parametrised by lattices  with periods $\omega_1=\tau, \omega_2=1$, in the following examples we will restrict to this case and write $\Lambda(\tau)$ for $\Lambda(\tau,1)=\ZZ\tau + \ZZ$, $R(\tau)$ for $R(\tau,1)$, and $T(\tau)$ for $T(\tau,1)$.

\begin{example}\label{ex.i}
In the case $\tau=i$, $R(i)$ is the ring  of Gaussian integers $\ZZ[i]$ and so it coincides with $\Lambda(i)$ (see  \cite[Example~IV~4.20.1]{Har:alg}). Thus the truss of endomorphisms of the curve $\eE = \CC/\Lambda(i)$  can be identified with $\ZZ\times\ZZ\times \eE$ with operations
$$
\Big[(m,n,[a]),(m',n',[a']),(m'',n'',[a''])\Big] = (m-m'+m'', n-n'+n'', [a-a'+a'']),
$$
$$
(m,n,[a],)(m',n',[a']) = (mm' - nn', mn'+nm', [(m+in)a' +a]).
$$
\hfill $\blacktriangle$
\end{example}

\begin{example}\label{ex.2i}
In the case $\tau=2i$, $R(2i)$  coincides with the subring of the ring of integers $\ZZ[i]$ of the quadratic field $\QQ(i)\subset \CC$ with conductor 2, that is 
$$
R(i) = \ZZ+2\ZZ[i]=\{m+2ni \; | m,n\in \ZZ\},
$$  
(see \cite[Example~IV~4.20.3]{Har:alg}). Thus $T(\CC/\Lambda(2i))$  can be identified with $\ZZ\times 2\ZZ\times \eE$ with the Cartesian product heap operations
 as in Example~\ref{ex.i}
 and multiplication
$$
\begin{aligned}
(m,2n,[a])(m',2n',[a']) &= \Big(mm'-4nn', 2(mn'+nm'), \left[(m+2ni)a' +a\right]\Big).
\end{aligned}
$$
\hfill $\blacktriangle$
\end{example} 

\begin{example}\label{ex.com}
Examples~\ref{ex.i}--\ref{ex.2i} describe curves with complex multiplication, i.e.\ such that $R(\tau)$ is strictly bigger than $\ZZ$. By \cite[Theorem~IV~4.19]{Har:alg} this is the case if and only if $\tau = p+q\sqrt{-d}$ where $p,q\in \QQ$ and $d$ a is positive integer, and then
$$
R(\tau) = \{m+n\tau \;|\; m,n\in \ZZ,\; \mbox{such that} \; 2np,n(p^2+dq^2)\in \ZZ\}\subseteq \Lambda(\tau).
$$
For all $p,q\in \QQ$ and positive integers $d$, let us define the additive subgroup (hence a sub-heap) of $\ZZ$,
$$
\ZZ(p,q,d) = \{n\in \ZZ\; |\; 2np\; \&\; n(p^2+dq^2)\in \ZZ\}.
$$
Then the multiplication in the truss
$$
T(\CC/\Lambda(p+q\sqrt{-d}))\cong \ZZ \times \ZZ(p,q,d)\times \CC/\Lambda(p+q\sqrt{-d})
$$
 comes out as
$$
\begin{aligned}
(m,n &,[a])(m',n',[a'])\\
&= \Big(mm'-nn'(p^2+dq^2), n+n'+2nn'p, 
\Big[\left(m+n\left(p+q\sqrt{-d}\right)\right)a' + a\Big]\Big).
\end{aligned}
$$
\hfill $\blacktriangle$
\end{example}

\begin{example}\label{ex.omega}
Let $\tau$ be the third root of unity with the positive imaginary part.  Since then $\tau = - \frac 12 + \frac12 \sqrt{-3}$, we find ourselves in the setup of Example~\ref{ex.com} with $p= - \frac 12$, $q= \frac12$ and $d=-3$. One easily computes that
 $$
\ZZ\Big( \frac 12, \frac12,-3\Big) = 2\ZZ.
$$
 Therefore, the truss of endomorphisms of the curve $\eE = \CC/\Lambda(\tau)$  comes out as $T(\tau) = \ZZ\times 2\ZZ\times \eE$ with multiplication
$$
\begin{aligned}
(m,2n,[a])(m',2n'&,[a']) \\
&= \Big(mm'+2nn', 2(n+n' - nn'), 
\Big[\left(m-n+ n\sqrt{-3}\right)a' + a\Big]\Big).
\end{aligned}
$$
\hfill $\blacktriangle$
\end{example} 

\section{In other fields}
Of course the construction of a truss of endomorphisms of an elliptic curve can be performed for the general non-complex case even though curves can be no longer identified with tori. Let $\FF$ be a (perfect) field and let $\eE$ be a smooth curve of genus one with a non-empty set of rational points $\eE(\FF)$\footnote{These are all the points of $\eE$ that are solutions in $\FF$ to a polynomial equation with coefficients from $\FF$.} (and hence an elliptic curve over $\FF$). By using the Riemann-Roch theorem and mapping $\eE$ into the projective plane $\PP^2$, $\eE$ can be represented by a cubic equation in the Weierstrass form. By the B\'ezout theorem every line through two points in $\eE$ crosses $\eE$ at the third point, and hence the geometric construction of the heap operation 
$$
[-,-,-] : \eE\times\eE\times\eE \longrightarrow \eE, \qquad (A,B,C)\longmapsto [A,B,C],
$$
described at the beginning of Section~\ref{sec.result} can be repeated verbatim, thus leading to the abelian heap $H(\eE)$.

\begin{lemma}\label{lem.gen}
Every endomorphism of an elliptic curve $\eE$ over $\FF$ is an endomorphism of the heap $H(\eE)$. Consequently, endomorphisms of $\eE$ form a truss $T(\eE)$.
\end{lemma}
\begin{proof}
We can retract  $H(\eE)$ to a group $G(\eE;O)$ at any rational point $O\in \eE(\FF)$, so that $A+B = [A,O,B]$. Any endomorphism $f$ of $\eE$ is an isogeny $\phi$ of $\eE$ with respect to $O$, that is an endomorphism of $G(\eE;O)$, combined with the translation by $f(O)$ (see e.g.\ \cite[Chapter III~Example~4.7]{Sil:ari}). The isogeny $\phi$ is fully determined by $f$. Explicitly,
$$
f(A) = [\phi(A),O,f(O)] =\phi(A)+f(O), \qquad \phi(A) = [f(A),f(O),O]=f(A)-f(O).
$$
Since $[A,B,C] = A-B+C$, and $\phi$ is an endomorphism of $G(\eE;O)$ one easily checks that $f$ is a heap endomorphism.
\end{proof}

Similarly to the complex curve case the truss $T(\eE)$ can be interpreted as a crossed product. The group $G(\eE;O)$ is a left module over the ring $R(\eE;O)$ of all isogenies of $\eE$ at $O$ by evaluation, $(\phi,A)\mapsto \phi(A)$.  Hence $R(\eE;O)\times \eE$ is a truss with the Cartesian product heap operation and multiplication
$$
(\phi, A)(\psi, B) = (\phi\circ \psi, [A,O,\phi(B)])
$$
isomorphic with $T(\eE)$ by the map 
$$
R(\eE;O)\times \eE \longrightarrow T(\eE), \qquad  (\phi,A)\longmapsto \Big[B\mapsto [\phi(B),O,A]\Big].
$$
The product in the  ring associated to $R(\eE;O)\times \eE$  at the point $(c_O:A\mapsto O, O)$ as in Remark~\ref{rem.ring} comes out as
$$
(\phi, A)\bullet (\psi, B) = (\phi\circ \psi, \phi(B)).
$$

\section*{Acknowledgements} 
I would like to thank Laiachi El Kaoutit first for raising the question about the additive structure of elliptic curves during the conference of {\em Hopf algebras, monoidal categories and related topics}, Bucharest, July 27--29, 2022, and second for the subsequent comments and suggestions.  I would also like to thank Simion Breaz and Bernard Rybo\l owicz for their comments and suggestions.

This research is partially supported by the National Science Centre, Poland, grant no. 2019/35/B/ST1/01115.

\end{document}